\documentclass[11pt,oneside]{article}
\usepackage{amsfonts,amsmath,amssymb,amsthm}
\usepackage{footmisc,fancyhdr,graphicx,float}
\usepackage{natbib}
\usepackage[colorlinks=true,linkcolor=blue,citecolor=blue,urlcolor=magenta,pdfstartview=FitB,bookmarksopen=true]{hyperref}
\usepackage[top=1.1in,bottom=1.1in,left=1.1in,right=1.1in]{geometry}
\usepackage[onehalfspacing]{setspace}

\usepackage{tikz}
\usetikzlibrary{trees,arrows,automata,graphs}
\usetikzlibrary{positioning,chains,fit,shapes,calc,topaths}
\usetikzlibrary{calc,decorations.markings}

\usepackage{hyphenat}
\newtheorem*{conjecture*}{Conjecture}
\newtheorem{theorem}{Theorem}

\newtheorem{corollary}[]{Corollary}

\newtheorem{example}[]{Example}

\newtheorem{proposition}[]{Proposition}

\newcommand{\mbb}{\mathbb}
\newcommand{\mc}{\mathcal}

\newcommand{\rbk}[1]{\left(#1\right)}

\begin{document}

\title{\vspace{-1.5cm}Resolution of a conjecture on variance functions for one-parameter natural exponential family}
\author{Xiongzhi Chen\footnote{Center for Statistics and Machine Learning, and Lewis-Sigler Institute for Integrative Genomics, Princeton University, Princeton, NJ 08544. Email: \texttt{xiongzhi@princeton.edu}}}
\date{}
\maketitle

\begin{abstract}
  One-parameter natural exponential family (NEF) plays fundamental roles in probability and statistics. A conjecture of Bar-Lev, Bshouty and Enis states that a polynomial with a simple root at $0$ and a complex root with positive imaginary part is the variance function of some NEF with mean domain $\rbk{0,\infty}$ if and only if the real part of the complex root is not positive. We prove the conjecture in this note.
\end{abstract}

\section{Introduction}\label{Sec:Intro}

One-parameter natural exponential family (NEF) is an important family of probability measures
and has been widely used in statistical modelling \citep{McCullagh:1989,Letac:1992}.
Let $\beta$ be a positive Radon measure on $\mathbb{R}$ that is not concentrated on one point. Suppose the interior $\Theta$ of
\begin{equation*}
\tilde{\Theta}=\left\{  \theta\in\mathbb{R}:g\left(  \theta\right) =\int%
e^{x\theta}\beta\left(  dx\right)  <\infty\right\}%\label{q7}
\end{equation*}
is not empty. Then $\mathcal{F}=\left\{F_{\theta}  :\theta\in\Theta\right\}  $ with
\begin{equation*}
F_{\theta}\left(dx\right)=\exp\left[  \theta x-\log g\left(  \theta\right)  \right]
\beta\left(  dx\right)  %\label{q2}
\end{equation*}
forms a NEF with respect to the basis $\beta$. Without loss of generality, we can assume $0 \in \Theta$, so that  $\beta\rbk{\mbb{R}}=1$.

For each NEF $\mathcal{F}$, the mapping $\mu\left(  \theta\right) =\dfrac{d}{d\theta}\log g\left(  \theta\right)$ defines the mean function $\mu:\Theta\rightarrow U$ with $U=\mu\left(  \Theta\right)  $, and $U$ is called the ``mean domain''. Let $\theta=\theta\left(  \mu\right)  $ be the inverse function of $\mu$. Define
the function $v$ on $U$ as
\begin{equation*}
  v\left(  \mu\right)  =\int\left(x-\mu\right)  ^{2}F_{\theta\left(  \mu\right)  }\left(  dx\right)%\label{eq:varfunc}
\end{equation*}
for $\mu\in U$. Then the pair $\left(  v,U\right)  $ is called the variance function (VF) of $\mathcal{F}$. Note that $v$ is a positive, real analytic function on $U$ and that $v$ characterizes the NEF $\mathcal{F}$.

Let $\mc{G}_D$ be the set of positive, real analytic functions on a domain $D \subseteq \mbb{R}$. In order to use an $f \in \mc{G}_D$ to model the variance-mean relationship for some data and assume that the associated random errors follow a NEF, it is crucial to ensure that $\rbk{f,D}$ is indeed the VF of some NEF. Therefore, identifying which $f \in \mc{G}_D$ are VFs of NFEs is of importance. For few of these results, see, e.g., \cite{Lev:1991,Letac:1992,Letac:2016}.

The authors of a ground breaking paper \cite{Lev:1992} raised the following:
\begin{conjecture*}\label{Conj}
Let
\begin{equation}\label{eq:vf}
v\left(  u\right)  =a_0 u\left(  u-u_{1}\right)^{n}\left(  u-\bar{u}_{1}\right)  ^{n}\text{,}
\end{equation}
where $a_{0}>0$, $\Im\rbk{u_{1}}>0$, and $\bar{u}_{1}$ is the complex conjugate of $u_{1}$. Then $\left(
v,\left(  0,\infty\right)  \right)  $ is a variance function for all $n\in \mathbb{N}$ if
and only if $\Re \rbk{u_{1}}\leq0$.
\end{conjecture*}

Let $\mathbb{R}^{+}=\left(  0,\infty\right)  $.
When $\Re\left(  u_{1}\right)  \leq0$ the polynomial $v$ has non-negative
coefficients. By \cite{Lev:1987} and Proposition 4.4 of \cite{Letac:1990}, $\left(
v,\mathbb{R}^{+}\right)  $ is a VF of an infinitely divisible NEF concentrated
on the set $\mathbb{N}$ of non-negative integers. This justifies the sufficiency in the conjecture, and we are generating
probability distributions on $\mathbb{N}$ using $v$. Thus, it is left to prove the necessity.
\cite{Lev:1992} derived a powerful result, their Theorem 1, for this purpose.
Specifically, this theorem requires to determine the sign of $\Re\rbk{\tau}$ of the residue $\tau=-2\pi i\operatorname*{Res}\left(  1/v,u_{1}\right)  $, where $i=\sqrt{-1}$. \cite{Lev:1992} directly computed $\tau$ to determine the sign of $\Re\rbk{\tau}$ and only verified the necessity for $n\in\left\{  1,2,3\right\}  $.
However, when $n$ is large computing $\tau$ becomes intractable since $1/v$ is the reciprocal of a polynomial $v$ of degree $2n+1$. This is revealed by the already complicated expression for $\tau$ when $n=3$ in their Theorem 6.
So a general, abstract approach is needed to determine the sign of $\Re\rbk{\tau}$ in order to prove the necessity for all $n$.

In this note, we prove the necessity for all $n \in \mbb{N}$ using a different strategy that exploits the algebraic and analytic properties of $v$. The positive answer to the conjecture enlarges the family of polynomials that can be the VFs of NEFs and the variance-mean relationships that can be used for statistical modelling using NEFs.

\section{Proof of necessity}\label{Sec:ResolveConj}

Let $\mu_{0}=\mu\left(  0\right)$ and $\theta_{0}=\lim\limits_{u\rightarrow+\infty}\int_{\mu_{0}}^{u}\frac
{dt}{v\left(  t\right)  }$. Then $0<\theta_{0}<\infty$ when $\left(
v,\mathbb{R}^{+}\right)  $ is a VF and $\Theta=\left(  -\infty,\theta
_{0}\right)  $. Let $d=$ $\frac{2\pi}{v^{\prime}\left(  0\right)  }$,
$\Theta+2^{-1}id=\left\{  z\in\mathbb{C}:z=\theta+2^{-1}id,\theta\in
\Theta\right\}  $ and
\[
S_{0}^{+}=\left\{  z\in\mathbb{C}:\Re\left(  z\right)  \in\Theta\text{ and
}0<\Im\left(  z\right)  <2^{-1}d\right\}.
\]
Our proof relies on a partial result of Theorem 1 of
\cite{Lev:1992} rephrased in our notations as:

\begin{proposition}
[]\label{Thm:BBE}Let $v$ be a polynomial of degree greater than $2$ such
that it has a simple zero at the origin, only one zero with positive imaginary
part and its conjugate, and no other zeros. Then, if $\left(  v,\mathbb{R}^{+}\right)  $ is a VF
of some NEF, it is necessary that
$\theta_{0}+\tau\notin S_{0}^{+}$ and $\theta_{0}+\tau\notin\Theta +2^{-1}id$.

\end{proposition}

Instead of computing $\tau$ to determine the sign of $\Re\left(  \tau\right)
$, our strategy is to show $\Im\left(  \tau\right)=2^{-1}d$ and
directly establish the correspondence such that
\begin{equation}
\Re\left(  \tau\right)  <0\text{\ whenever \ }\Re\left(  u_{1}\right)
>0\text{\ for all \ }n\in\mathbb{N}.\label{eq:key}
\end{equation}
From this, an application of \autoref{Thm:BBE} gives $\Re\left(  \tau\right)  \geq0$ when $\left(
v,\mathbb{R}^{+}\right)  $ is a VF, which by \eqref{eq:key} forces $\Re\left(
u_{1}\right)  \leq0$ and yields the necessity. We now proceed to the proof.

% insert graph
\begin{figure}[!]
\centering
% contour
%\resizebox{2.8in}{1.3in}{
\begin{tikzpicture}[scale=0.85,transform shape]
\draw[->] (0,-0.5) -- (0,5.5);  % Axis
\draw[->] (-5.5,0) -- (5.5,0);

\node at (0.3,-.3) {$0$}; % origin
\node at (5.5,-.3) {$x$}; % x
\node at (0.3,5.5) {$iy$}; % iy
\node at (-1.5,2) {$u_{1}$};  % symbol

% labels for contour
\node at (2,5) {$\Gamma_1$}; %
\node at (0.4,0.7) {$\Gamma_2$};
\node at (-2.5,0.3) {$\Gamma_3$};
\node at (2,0.3) {$\Gamma_4$};

% Contour line
\draw[thick,red,yshift=2pt,
decoration={ markings,  % This schema allows for fine-tuning the positions of arrows
      mark=at position 0.1 with {\arrow[line width=1pt]{>}},
      mark=at position 0.3 with {\arrow[line width=1pt]{>}},
      mark=at position 0.5 with {\arrow[line width=1pt]{>}},
      mark=at position 0.8 with {\arrow[line width=1pt]{>}},
      mark=at position 0.98 with {\arrow[line width=1pt]{>}}},
      postaction={decorate}]
  (-5,0) -- (-0.5,0) arc (180:0:.5) -- (0.5,0) -- (5,0); % line segment (-5,0) to (-0.5,0), connect
                                             % to clockwise semi-circle, connect to (5,0)

\draw[thick,red,yshift=2pt,
decoration={ markings,
      mark=at position 0.2 with {\arrow[line width=1pt]{>}},
      mark=at position 0.4 with {\arrow[line width=1pt]{>}},
      mark=at position 0.6 with {\arrow[line width=1pt]{>}},
      mark=at position 0.8 with {\arrow[line width=1pt]{>}}},
      postaction={decorate}]
(5,0) arc (0:180:5); % semi-circle, counter clockwise, starting at (5,0)

% draw the poles
\draw [fill] (0,0) circle [radius=2pt];
\draw [fill] (-2,2) circle [radius=2pt]; %u_1
\end{tikzpicture}
%} % end of resizebox
\caption[]{Contour $\Gamma$ in the proof of \autoref{Thm:Done}, where the arrow indicates orientation.}
\label{FigContour}
\end{figure}
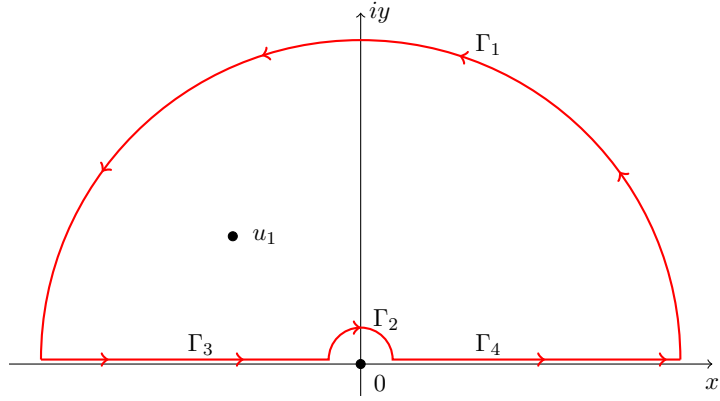
% end of insert graph

\begin{theorem}
\label{Thm:Done}The necessity is true for all $n\in \mathbb{N}$.
\end{theorem}

\begin{proof}
Without loss of generality, we can assume $a_0=1$.
Pick $\rho$ and $R$ such that $5^{-1}\left\vert u_{1}\right\vert >\rho>0$ and
$R>2\left\vert u_{1}\right\vert $. Define the contour $\Gamma=\bigcup
\nolimits_{i=1}^{4}\Gamma_{i}$ depicted in \autoref{FigContour}, where
$\Gamma_{1}=\left\{  Re^{i\omega}:0\leq\omega\leq\pi\right\}  $, $\Gamma
_{2}=\left\{  \rho e^{i\omega}:\pi\geq\omega\geq0\right\}  $, $\Gamma
_{3}=\left\{  x:-R\leq x\leq-\rho\right\}  $ and $\Gamma_{4}=\left\{
x:\rho\leq x\leq R\right\}  $. Then $\tau=-\int_{\Gamma}\frac{d\zeta}{v\left(
\zeta\right)  }$. However, $\int_{\Gamma_{2}}\frac{d\zeta}{v\left(
\zeta\right)  }=\frac{-i\pi}{v^{\prime}\left(  0\right)  }$, $\lim
\limits_{R\rightarrow\infty}\left\vert \int_{\Gamma_{1}}\frac{d\zeta}{v\left(
\zeta\right)  }\right\vert \leq\lim\limits_{R\rightarrow\infty}\frac{\pi}{6R^{n+1}}%
=0$, and $L_{R}=\int_{\Gamma_{3}\cup\Gamma_{4}}\frac{d\zeta}{v\left(
\zeta\right)  }$ is real. So, $\Im\left(  \tau\right)  =\frac{\pi}{v^{\prime
}\left(  0\right)  }=2^{-1}d$ and
\begin{equation}
\Re\left(  \tau\right)  =-\lim_{R\rightarrow\infty}L_{R}=-\lim_{R\rightarrow
\infty}\int_{\rho}^{R}\alpha\left(  \zeta\right)  d\zeta\text{,}\label{eq:Rep}%
\end{equation}
where
\begin{equation}
\alpha\left(  z\right)  =\frac
{1}{z}\left(  \frac{1}{\left\vert z-u_{1}\right\vert ^{2n}}-\frac
{1}{\left\vert z+u_{1}\right\vert ^{2n}}\right)=
\frac{\vartheta\left(  -z\right)  -\vartheta\left(
z\right)  }{z\vartheta\left(  -z\right)  \vartheta\left(  z\right)  }  \label{eq:integrand}
\end{equation}
and%
\begin{equation}
\vartheta\left(  z\right)  =\left(  z^{2}-2z\Re\left(  u_{1}\right)
+\left\vert u_{1}\right\vert ^{2}\right)  ^{n}\text{.}\label{eq:hattheta}
\end{equation}

It suffices now to show that $\Re\left(  u_{1}\right)  >0$ is untenable. Given
$\Re\left(  u_{1}\right)  >0$, then $\alpha\left(  z\right)  >0$ for
$z\geq\rho$ and $\inf_{z\in\left[  \rho,\rho_{1}\right]  }\alpha\left(
z\right)  >0$ for any finite $\rho_{1}\geq\rho$. So, $\Re\left(
u_{1}\right)  >0$ implies $\Re\left(  \tau\right)  <0$, i.e., (\ref{eq:key})
holds. Let $\theta_{1}=\theta_{0}+\tau$. Then $\theta_{1}
=\theta_{0}+\Re\left(  \tau\right)  +2^{-1}id$. Since $\left(
v,\mathbb{R}^{+}\right)  $ is a VF, \autoref{Thm:BBE} forces
$\theta_{1}\notin\Theta+2^{-1}id$ and $\Re\left(  \tau\right)  \geq0$,
a contradiction. The proof is complete.
\end{proof}

\section*{Acknowledgements}

This research is funded by the Office of Naval Research grant N00014-12-1-0764.
I am very grateful to the Associate Editor and two anonymous referees for valuable suggestions
and comments. I would like to thank John D. Storey, Daoud Bshouty, G\'{e}rard Letac
and Persi Diaconis for support, comments, discussions and encouragements.

%%%% bib
%\bibliographystyle{dcu}
%\bibliography{SpecialExpFam}

%%%%%%%%%%%% second paper %%%%%%%%%%%%%%%
\pagebreak
\setcounter{page}{1}
\renewcommand*{\thefootnote}{\fnsymbol{footnote}}
\setcounter{section}{0}
\setcounter{theorem}{0}
\setcounter{equation}{0}

\begin{center}
\title{\vspace{-.5cm}{\LARGE Reduction functions for the variance function of one-parameter\\
\bigskip
natural exponential family}}
\bigskip
\medskip

\author{Xiongzhi Chen\footnote{Center for Statistics and Machine Learning, and
Lewis-Sigler Institute for Integrative Genomics, Princeton University,
Princeton, NJ 08544. Email: \texttt{xiongzhi@princeton.edu}}}
\date{}
\maketitle
\bigskip

\begin{abstract}
In this note, we consider the problem of constructing an unbiased estimator of
the variance of a random variable $\xi$ from a single realization. We show
that, when $\xi$ has parametric distributions that form an infinitely
divisible (i.d.) natural exponential family (NEF) whose induced measure is
absolutely continuous with respect to the basis measure of the NEF, there
exists a deterministic function $h$, called ``reduction function'', such that
$h\left(  \xi\right)  $ is an unbiased estimator of the variance function of
$\xi$. In particular, such an $h$ exists when $\xi$ follows an i.d. NEF
concentrated on the set of nonnegative integers, a NEF with cubic variance
function, or a NEF with power variance function. The utility of the reduction
function is illustrated by an application to estimating the latent linear
space in high-dimensional data.
\medskip
\newline\textit{Keywords}: Natural
exponential family, reduction function, variance function.
\medskip
\newline\textit{MSC 2010 subject classifications}: Primary 60E07; Secondary 62E10.

\end{abstract}
\end{center}

\section{Introduction}

\label{Sec:Intro}

An important task in statistical inference is to construct of a simple,
unbiased estimator of the variance of a random variable (rv). Usually such an
estimator is obtained from multiple realizations of the rv. In the paper, we
explore the possibility of obtaining such an estimator from a transformation
of a single realization of the rv. One motivation behind our study is the need
to adjust for the variances in order to estimate the latent space that
generates the means of the observations; see Theorem 2 of
\cite{Chen:2014intra}. From a theoretical perspective, our study is
essentially classifying probability measures with special variance-mean
relationship (VMR).

Let $\mathbb{E}$ and $\mathbb{V}$ be the mean and variance operators. Stated
formally, we attempt to answer the following the question: for which rv $\xi$
with finite variance does there exist a function $\varphi$ such that
$\mathbb{E}\left[  \varphi\left(  \xi\right)  \right]  =\mathbb{V}\left[
\xi\right]  $? When $\varphi$ exists, it is called a \textquotedblleft
reduction function (RF)\textquotedblright. It can be perceived that the
existence of $\varphi$ depends crucially on the VMR of $\xi$. For example,
when $\xi$ is a Poisson rv $\varphi$ exists and is the identity. However, when
$\xi$ has a different distribution, $\varphi$ may not exist at all.
Unfortunately, a complete answer to the question seems to be out of reach. So,
we focus on the case where $\xi$ has a parametric distribution in the form of
some natural exponential family (NEF) $\mathcal{F}$. This allows partial but
useful results on the existence of $\varphi$ due to the richness and wide
usage of NEFs in probability and statistics \citep{McCullagh:1989,Letac:1992}.

We show that $\varphi$ always exists if the NEF $\mathcal{F}$ is infinitely
divisible (i.d.) and an associated convolution of measures is absolutely
continuous (a.c.) with respect to (wrt) the basis measure of $\mathcal{F}$;
see \autoref{ThmMain}. Our results imply that $\varphi$ exists if $\xi$
follows a NEF with cubic variance function (NEF-CVF, \citealp{Letac:1990}), a
NEF with polynomial variance function (NEF-pVF, \citealp{Lev:1992}), or a NEF
with power variance function (NEF-PVF, \citealp{Lev:1986}). These NEFs cover
many distributions used in statistical modelling.
%(which includes the Tweedie family in \cite{Tweedie:1984})

In its most general form, the existence of $\varphi$ for a NEF $\mathcal{F}$
with basis measure $\beta$ and maximal open parameter space $\Theta$ is
equivalent to the existence of $\varphi$ satisfying the integro-differential
equation%
\[
\int\varphi\left(  x\right)  e^{\theta x}\beta\left(  dx\right)  =\left[
\dfrac{d^{2}}{d\theta^{2}}\log\int e^{\theta x}\beta\left(  dx\right)
\right]  \int e^{\theta x}\beta\left(  dx\right)  \text{, }\forall\theta
\in\Theta.
\]
However, this is a very challenging problem whose complete answer is left for
future research.

\section{Preliminaries on NEFs}

\label{sec:onNEFs}

We review the definition of a NEF and some properties of its mean and variance
functions, which can be found in \cite{Letac:1992}. Let $\beta$ be a positive
Radon measure on $\mathbb{R}$ that is not concentrated on one point. Let
$L\left(  \theta\right)  =\int e^{x\theta}\beta\left(  dx\right)  $ for
$\theta\in\mathbb{R}$ be its Laplace transform and $\Theta$ be the maximal
open set containing $\theta$ such that $L\left(  \theta\right)  <\infty$.
Suppose $\Theta$ is not empty. Then
\[
\mathcal{F}=\left\{  F_{\theta}:F_{\theta}\left(  dx\right)  =\exp\left[
\theta x-\log L\left(  \theta\right)  \right]  \beta\left(  dx\right)
,\theta\in\Theta\right\}
\]
forms a NEF with respect to the basis $\beta$. Without loss of generality
(WLOG), we can assume $0\in\Theta$, so that $\beta\left(  \mathbb{R}\right)
=1$.

Let $\kappa\left(  \theta\right)  =\log L\left(  \theta\right)  $ be the
cumulant function of $\beta$. For the NEF $\mathcal{F}$, the mapping
\begin{equation}
\mu\left(  \theta\right)  =\int xF_{\theta}\left(  dx\right)  =\kappa^{\prime
}\left(  \theta\right)  \label{eq:mean}%
\end{equation}
defines the mean function $\mu:\Theta\rightarrow U$ with $U=\mu\left(
\Theta\right)  $, and $U$ is called the \textquotedblleft mean
domain\textquotedblright. Further, the mapping
\begin{equation}
V\left(  \theta\right)  =\int\left(  x-\mu\left(  \theta\right)  \right)
^{2}F_{\theta}\left(  dx\right)  =\kappa^{\prime\prime}\left(  \theta\right)
\label{eq:var}%
\end{equation}
defines the variance function (VF). Let $\theta=\theta\left(  \mu\right)  $ be
the inverse function of $\mu$. Then $V$ can be parametrized by $\mu$ as
\[
V\left(  \mu\right)  =\int\left(  x-\mu\right)  ^{2}F_{\theta\left(
\mu\right)  }\left(  dx\right)  \text{ for }\mu\in U,
\]
and the pair $\left(  V,U\right)  $ is called the VF of $\mathcal{F}$.

\section{Reduction functions for some infinitely divisible NEFs}

\label{Sec:Representation}

Let $\xi_{\theta}$ with $\theta\in\Theta$ have distribution $F_{\theta}%
\in\mathcal{F}$, then its mean is $\mu\left(  \theta\right)  $ and variance
$V\left(  \theta\right)  $ as defined by (\ref{eq:mean}) and (\ref{eq:var}).
We will identify some $\mathcal{F}$ for which there exits a function $\varphi$
independent of $\theta$ such that
\begin{equation}
\mathbb{E}\left[  \varphi\left(  \xi_{\theta}\right)  \right]  =\mathbb{V}%
\left[  \xi_{\theta}\right]  ,\text{ \ }\forall\theta\in\Theta\text{.}
\label{q3}%
\end{equation}
Note that (\ref{q3}) is equivalent to%
\begin{equation}
\int\varphi\left(  x\right)  e^{x\theta}\beta\left(  dx\right)  =L\left(
\theta\right)  \kappa^{\prime\prime}\left(  \theta\right)  \text{, }%
\forall\theta\in\Theta. \label{q3A}%
\end{equation}

Let $\delta_{y}$ be the Dirac mass at $y\in\mathbb{R}$. We provide simple
sufficient conditions on the existence of $\varphi$.

\begin{proposition}
\label{ThmMain}$\mathcal{F}$ is an i.d. NEF iff $\kappa^{\prime\prime}\left(
\theta\right)  =\int e^{\theta x}\rho\left(  dx\right)  $ where%
\begin{equation}
\rho\left(  dx\right)  =\sigma^{2}\delta_{0}\left(  dx\right)  +x^{2}%
\nu\left(  dx\right)  \label{eqRho}%
\end{equation}
and the L\'{e}vy triple $\left(  \epsilon,\sigma^{2},\nu\right)  $ is given in
(\ref{eqLevy}). Therefore, if $\mathcal{F}$ is i.d., then setting
$\alpha=\beta\ast\rho$ gives the equivalent to (\ref{q3A}) as
\begin{equation}
\int\varphi\left(  x\right)  e^{\theta x}\beta\left(  dx\right)  =\int
e^{\theta x}\alpha\left(  dx\right)  . \label{eqEquiv}%
\end{equation}
Further, if $\alpha\ll\beta$, then $\varphi$ exits as the Radon-Nikodym
derivative $\varphi=\frac{d\alpha}{d\beta}$.
\end{proposition}

\begin{proof}
Since $0\in\Theta$, then $\beta\in\mathcal{F}$. By Proposition 4.1 of
\cite{Letac:1992}, $\beta$ is i.d iff $\mathcal{F}$ is. By Theorem 6.2 of
\cite{Letac:1992}, $\beta$ is i.d. iff there exist some constants $\epsilon$
and $\sigma^{2}\geq0$ and L\'{e}vy measure $\nu$ such that
\begin{equation}
\kappa\left(  \theta\right)  =\epsilon\theta+2^{-1}\sigma^{2}\theta^{2}%
+\int_{\mathbb{R}\setminus\left\{  0\right\}  }\left(  e^{\theta x}%
-1-\theta\tau\left(  x\right)  \right)  \nu\left(  dx\right)  \label{eqLevy}%
\end{equation}
for a centering function $\tau:\mathbb{R}\rightarrow\mathbb{R}$ such that
$x^{-2}\left(  \tau\left(  x\right)  -x\right)  $ is bounded as $x\rightarrow
0$. Again, by Theorem 6.2 of \cite{Letac:1992}, $\beta$ is i.d. iff
$\kappa^{\prime\prime}\left(  \theta\right)  =\int e^{\theta x}\rho\left(
dx\right)  $ with $\rho$ given by (\ref{eqRho}). Obviously, $\int_{\left\{
\left\vert x\right\vert \geq1\right\}  }x^{-2}\rho\left(  dx\right)  <\infty$
since $\nu$ is a L\'{e}vy measure. Therefore, $\kappa^{\prime\prime}\left(
\theta\right)  g\left(  \theta\right)  =\int e^{\theta x}\alpha\left(
dx\right)  $, and (\ref{q3A}) becomes (\ref{eqEquiv}). When $\alpha\ll\beta$,
we see that $\varphi=\frac{d\alpha}{d\beta}$ by Radon-Nikodym theorem (e.g.,
in \citealp{Halmos:1950}). This completes the proof.
\end{proof}

We call $\rho$ the measure induced by the i.d. NEF $\mathcal{F}$. By
\autoref{ThmMain}, to find $\varphi$ we first need to check if a NEF is i.d.,
and if so, we then check if $\alpha\ll\beta$. We will focus on i.d. NEFs whose
VFs are simple but flexible for statistical applications rather than
exhausting all i.d. NEFs. \cite{Lev:1987} and \cite{Chen:2015expfam} showed
that a NEF-pVF with $V\left(  u\right)  =\sum_{k=1}^{n}a_{k}u^{k}$ and mean
domain $\mathbb{R}_{+}:=\left(  0,\infty\right)  $ is i.d. when all $a_{k}$'s
are nonnegative. In particular, such NEFs include six members of the NEF-CVF
that are concentrated on the set $\mathbb{N}$ of nonnegative integers; see
Table 2 in \cite{Letac:1990}.

We now show the existence and formula for the RF $\varphi$ for i.d. NEFs that
are concentrated on $\mathbb{N}$.

\begin{corollary}
\label{ThmDiscreteN}Suppose $\mathcal{F}$ is i.d. and concentrated on
$\mathbb{N}$. Then
\begin{equation}
c_{n}=\left.  \frac{d^{n}}{dz^{n}}\kappa\left(  \log z\right)  \right\vert
_{z=0} \label{eqseqcn}%
\end{equation}
is well defined for $n\geq1$, and the measure $\alpha$ on $\mathbb{N}$ with%
\begin{equation}
\alpha\left(  \left\{  n\right\}  \right)  =\sum\nolimits_{k=0}^{n}%
\beta\left(  \left\{  n-k\right\}  \right)  \left(  k+2\right)  \left(
k+1\right)  c_{k+2} \label{eqLawConv}%
\end{equation}
is well-defined. Further, $\varphi$ exists and $\varphi\left(  n\right)
=\frac{\alpha\left(  \left\{  n\right\}  \right)  }{\beta\left(  \left\{
n\right\}  \right)  }$ for $n\in\mathbb{N}$.
\end{corollary}

\begin{proof}
Since $\beta$ is concentrated on $\mathbb{N}$, we can write $\beta\left(
dx\right)  =\sum_{n=0}^{\infty}\beta_{n}\delta_{n}\left(  dx\right)  $ with
$\beta_{n}=\beta\left(  \left\{  n\right\}  \right)  $ for each $n\in
\mathbb{N}$. By the criterion on page 290 of \cite{Feller:1971A}, $\beta$ is
i.d. and concentrated on $\mathbb{N}$ iff there exists a non-negative sequence
$\left\{  c_{n}\right\}  _{n\geq1}$ such that $\sum_{n=1}^{\infty}c_{n}z^{n}$
has radius of convergence $R\geq1$ and that $\kappa\left(  \theta\right)
=\sum\nolimits_{n=0}^{\infty}c_{n}e^{n\theta}$ with\ $c_{0}=-\sum
\nolimits_{n=1}^{\infty}c_{n}$ for \ $\theta\in\Theta=\left(  -\infty,\log
R\right)  $. So, the sequence $\left\{  c_{n}\right\}  _{n\geq1}$ is given by
(\ref{eqseqcn}) and
\begin{equation}
\kappa^{\prime\prime}\left(  \theta\right)  =\sum_{n=0}^{\infty}\left(
n+2\right)  \left(  n+1\right)  c_{n+2}e^{n\theta},\text{ }\forall\theta
\in\Theta. \label{eqLawNcumulant}%
\end{equation}
\ Therefore,
\begin{equation}
\rho\left(  dx\right)  =\sum_{n=0}^{\infty}\left(  n+2\right)  \left(
n+1\right)  c_{n+2}\delta_{n}\left(  dx\right)  \label{eqRhoLawN}%
\end{equation}
satisfies $\kappa^{\prime\prime}\left(  \theta\right)  =\int e^{\theta x}%
\rho\left(  dx\right)  $. On the other hand,
\[
L\left(  \theta\right)  =\exp\left(  \kappa\left(  \theta\right)  \right)
=\sum_{n=0}^{\infty}\beta_{n}e^{n\theta},\text{ }\forall\theta\in\Theta
\]
and $\beta_{0}>0$. The measure $\alpha$ with $\alpha_{n}$ in (\ref{eqLawConv})
is well defined. Let $A=\left\{  n\in\mathbb{N}:c_{n}>0\right\}  $ and
$S=\left\{  n\in\mathbb{N}:\beta_{n}>0\right\}  $. Then $\left\{
n\in\mathbb{N}:\alpha_{n}>0\right\}  \subseteq S$ and $\alpha\ll\beta$. Thus,
$\varphi:\mathbb{N}\rightarrow$ $\mathbb{R}_{+}$ such that $\varphi\left(
n\right)  =\alpha_{n}/\beta_{n}$ is the RF. This completes the proof.
\end{proof}

\autoref{ThmDiscreteN} provides the generic formula to obtain the RF $\varphi$
for i.d. NEFs concentrated on $\mathbb{N}$. However, it is not always easy to
obtain the sequences $\left\{  c_{n}\right\}  _{n\geq1}$ and $\left\{
\beta_{n}\right\}  _{n\geq1}$ in the series expansions of $\kappa\left(
\theta\right)  $ and $L\left(  \theta\right)  $ for such a NEF in order to get
$\varphi$; see Examples 4 to 6 on NEF-CVFs in \autoref{secNEFwithCVF}. In the
next two subsections, we will provide the RF $\varphi$ for NEF-CVF and NEF-PVF.

\subsection{RFs for NEFs with cubic variance functions\label{secNEFwithCVF}}

\begin{example}
NEF-QVF. Recall that a NEF-QVF in \cite{Morris:1982} has VF $V\left(
u\right)  =a_{0}+a_{1}u+a_{2}u^{2}$ for some $a_{i}\in\mathbb{R}$ for
$i=0,1,2$. For such a NEF, $\varphi$ can be obtained by directly solving
(\ref{q3A}) without using \autoref{ThmDiscreteN}. In fact, (\ref{q3}) holds
with%
\[
\varphi\left(  t\right)  =\left(  1+a_{2}\right)  ^{-1}\left(  a_{0}%
+a_{1}t+a_{2}t^{2}\right)
\]
when $a_{2}\neq-1$. Note that $a_{2}=-1$ corresponds to a Bernoulli
distribution for which $\varphi$ does not exist. \autoref{TbNEFqvf} gives
$\varphi$ for NEF-QVF and is reproduced from \cite{Chen:2014intra}.
\end{example}

\begin{table}[tbp] \centering

\begin{tabular}
[c]{|l|l|l|l|l|}\hline
Distribution & $\theta$ & $\mathbb{V}[\xi]$ & Reduction function $\varphi(x)$
& Link function $\eta(\theta)$\\\hline
Normal$(\vartheta,1)$ & $\vartheta$ & 1 & $1$ & $\theta$\\
Poisson$(\vartheta)$ & $\vartheta$ & $\theta$ & $x$ & $\log(\theta)$\\
Binomial$(m,\vartheta)$ & $m\tau$ & $\theta-\theta^{2}/m$ & $(mx-x^{2})/(m-1)$
& ${\text{logit}}(\theta/m)$\\
NegBin$(m,\vartheta)$ & $m\tau/(1-\vartheta)$ & $\theta+\theta^{2}/m$ &
$(mx+x^{2})/(m+1)$ & $\log\left[  \theta/(m+\theta)\right]  $\\
Gamma$(m,\vartheta)$ & $m/\vartheta$ & $\theta^{2}/m$ & $x^{2}/(1+m)$ &
$-1/\theta$\\
GHS$(m,\vartheta)$ & $m\tau$ & $m+\theta^{2}/m$ & $(m^{2}+x^{2})/(1+m)$ &
$\arctan(\theta/m)$\\\hline
\end{tabular}
\caption{The reduction function $\varphi$  such that $\mathbb{E}[\varphi
(\xi)]=\mathbb{V}[\xi]$ when the parametric distribution of $\xi
$ forms $\mathcal{F}=\left\{
F_{\theta}:\theta\in\Theta\right\}  $, a NEF-QVF in \cite{Morris:1982}.
Here ${\text{logit}}(x)=\log\left(  x/\left(  1-x\right)\right)  $ for
$x\in(0,1)$, $m\in\mathbb{N}$ and it is bigger than $1$ for Binomial
distributions, and \textquotedblleft GHS\textquotedblright\ stands for
\textquotedblleft generalized hyperbolic secant distribution\textquotedblright
. Note that GHS is called ``hyperbolic cosine'' by \cite{Letac:1990}.}%
\label{TbNEFqvf}
\end{table}

For a NEF-CVF with $\deg V=3$, directly solving (\ref{q3A}) for $\varphi$ is
no longer preferred. However, when $\deg V=3$, the corresponding NEF is i.d.
with mean domain $U=\left(  0,\infty\right)  $. So, we will use
\autoref{ThmMain} or \autoref{ThmDiscreteN} to derive $\varphi$. The following
three NEF-CVFs are concentrated on $\mathbb{N}$ and generated by analytic
functions; see Proposition 4.3 of \cite{Letac:1990}.

\begin{example}
Inverse Gaussian. Its VF is $V\left(  u\right)  =u^{3}$ and
\[
\beta\left(  x\right)  =\left(  2\pi\right)  ^{-1/2}x^{-3/2}\exp\left(
-1/\left(  2x\right)  \right)  1_{\left(  0,\infty\right)  }dx.
\]
So, $\kappa\left(  \theta\right)  =-\sqrt{-2\theta}$ with $\theta<0$, and
$\rho\left(  dx\right)  =\left(  2\pi\right)  ^{-1/2}x^{1/2}1_{\left(
0,\infty\right)  }\left(  x\right)  dx$. Take $E\in\mathcal{B}\left(
\mathbb{R}_{+}\right)  $, then $\beta\left(  E\right)  =0$ iff the Lebesgue
measure of $E$ is $0$, i.e., $\chi\left(  E\right)  =0$. Since $\chi$ is
translation invariant, $\int_{0}^{\infty}1_{E}\left(  x+y\right)  \beta\left(
dx\right)  =0$ if $\beta\left(  E\right)  =0$ for any $y>0$. This implies
$\alpha=\beta\ast\rho$ satisfies $\alpha\ll\beta$. The RF $\varphi$ can be
found by computing the inverse Laplace transform $\hat{L}_{\kappa}$ of
$L\left(  z\right)  \kappa^{\prime\prime}\left(  z\right)  $ for $z$ such that
$\Re\left(  z\right)  =\theta\in\Theta$. Details are given below. Since%
\[
L\left(  z\right)  \kappa^{\prime\prime}\left(  z\right)  =\left(  -2z\right)
^{-3/2}\exp\left(  -\sqrt{-2z}\right)  ,
\]
we see%
\[
\left\vert L\left(  z\right)  \kappa^{\prime\prime}\left(  z\right)
\right\vert \leq\left(  2\left\vert z\right\vert \right)  ^{-3/2}%
\sim\left\vert \Im\left(  z\right)  \right\vert ^{-3}\text{ \ as \ }\left\vert
\Im\left(  z\right)  \right\vert \rightarrow\infty
\]
uniformly for $\theta$ in any compact subset of $\left(  -\infty,0\right)  $.
So, by Theorem 19a of \cite{Widder:1946} $\hat{L}_{\kappa}$ is well-defined.
In fact,%
\[
\hat{L}_{\kappa}\left(  x\right)  =\left(  2\pi\right)  ^{-1}\left(
\sqrt{2\pi}\exp\left(  -2^{-1}x^{-1}\right)  -\pi x^{-1/2}+\pi\Phi\left(
2^{-1/2}x^{-1/2}\right)  \right)  ,
\]
where $\Phi\left(  x\right)  =2\pi^{-1/2}\int_{0}^{x}\exp\left(
-t^{2}\right)  dt$. Therefore, $\varphi\left(  x\right)  =\frac{\hat
{L}_{\kappa}\left(  x\right)  }{f\left(  x\right)  }1_{\left(  0,\infty
\right)  }$ is the RF.
\end{example}

\begin{example}
Strict Arcsine. Its VF is $V\left(  u\right)  =u\left(  1+u^{2}\right)  $ and
$\beta\left(  dx\right)  =\sum_{n=0}^{\infty}p_{n}\left(  1\right)
\dfrac{\delta_{n}\left(  dx\right)  }{n!}$, where%
\begin{equation}
p_{2n}\left(  t\right)  =\prod_{k=0}^{n-1}\left(  t^{2}+4k^{2}\right)  \text{
\ and \ }p_{2n+1}\left(  t\right)  =t\prod_{k=0}^{n-1}\left(  t^{2}+\left(
2k+1\right)  ^{2}\right)  . \label{eqPcoef}%
\end{equation}
The generating function of $\beta$ is $f\left(  z\right)  =\exp\left(  \arcsin
z\right)  $ for $\left\vert z\right\vert \leq1$ and $\kappa\left(
\theta\right)  =\arcsin e^{\theta}$ with $\theta<0$. By \cite{kokonendji1994},
$\rho\left(  dx\right)  =\sum\nolimits_{n=1}^{\infty}\delta_{2n+1}\left(
dx\right)  $.
\end{example}

\begin{example}
Ressel family. Its VF is $V\left(  u\right)  =u^{2}\left(  1+\mu\right)  $ and
$\beta\left(  dx\right)  =\dfrac{x^{x}e^{-x}}{\Gamma\left(  x+2\right)
}1_{\left(  0,\infty\right)  }dx$. By Lemma 4.1 of \cite{kokonendji1994},
$\rho\left(  dx\right)  =\sum_{n=0}^{\infty}\frac{\left(  n+2\right)  !}%
{2!n!}\beta^{\ast\left(  n+2\right)  }$, where $\beta^{\ast m}$ means the
convolution of $\beta$ by itself $m$ times.
\end{example}

The next three NEF-CVFs with $\deg V=3$ are concentrated on $\mathbb{N}$ and
generated by a composition of two analytic functions; see Theorem 4.5 of
\cite{Letac:1990}. To compute $\rho$, we reformulate Theorem 4.5 of
\cite{Letac:1990} and Lemma 4.2 of \cite{kokonendji1994} as follows.

\begin{proposition}
\label{PropLetacKokonendji}Let $g\left(  z\right)  =\sum_{n=0}^{\infty}%
g_{n}z^{n}$ with radius of convergence $R\left(  g\right)  \geq1$ such that
$g_{n}\geq0$ and $g_{0}g_{1}\neq0$. Then there exits some $0<R\left(
h\right)  <R\left(  g\right)  $ such that $h\left(  w\right)  -wg\left(
h\left(  w\right)  \right)  =0$ implicitly defines an analytic function $h$
for $\left\vert w\right\vert <R\left(  h\right)  $. Let
\[
\beta_{n}=\frac{1}{\left(  n+1\right)  !}\left.  \left(  \frac{d}{dz}\right)
^{n}\left(  g\left(  z\right)  \right)  ^{n+1}\right\vert _{z=0}%
\]
If $\mu_{n}\geq0$ for all $n\in\mathbb{N}$, then the measure $\beta\left(
dx\right)  =\sum_{n=0}^{\infty}\beta_{n}\delta_{n}\left(  dx\right)  $ is
well-defined and
\[
g\left(  h\left(  w\right)  \right)  =\frac{h\left(  w\right)  }{w}=\sum
_{n=0}^{\infty}\beta_{n}w^{n}.
\]
Further, setting $w=e^{\theta}$ gives%
\begin{equation}
\kappa^{\prime\prime}\left(  \theta\right)  =\left(  1-\frac{hg^{\prime}}%
{g}\right)  ^{-3}\frac{h\left(  gg^{\prime}+hgg^{\prime\prime}-hg^{\prime
2}\right)  }{g^{2}}, \label{eqLTRho}%
\end{equation}
where the derivative is taken wrt to $\theta$.
\end{proposition}

Since $\kappa^{\prime\prime}\left(  \theta\right)  =\int e^{x\theta}%
\beta\left(  dx\right)  $, taking $h\left(  e^{\theta}\right)  $ in
(\ref{eqLTRho}) as the argument gives the generating function $f_{\rho}$ for
$\rho$ as
\[
f_{\rho}\left(  z\right)  =H_{\rho}\left(  h\left(  z\right)  \right)
=\sum_{n=0}^{\infty}w^{n}\rho\left(  \left\{  n\right\}  \right)
\]
and%
\begin{equation}
H_{\rho}\left(  x\right)  =x\left(  1-\frac{xg^{\prime}\left(  x\right)
}{g\left(  x\right)  }\right)  ^{-3}\left(  \frac{g^{\prime}\left(  x\right)
}{g\left(  x\right)  }+\frac{xg^{\prime\prime}\left(  x\right)  }{g\left(
x\right)  }-x\left(  \frac{g^{\prime}\left(  x\right)  }{g\left(  x\right)
}\right)  ^{2}\right)  . \label{eqGFrho}%
\end{equation}
Note that (\ref{eqGFrho}) is also given in \cite{kokonendji1994}. However,
computing $f_{\rho}^{\left(  n\right)  }\left(  0\right)  $ from
(\ref{eqGFrho}) in order to get $\rho\left(  \left\{  n\right\}  \right)  $
can be quite tedious.

\autoref{PropLetacKokonendji} provides $\rho$ and hence $\varphi$ for the following:

\begin{example}
Abel family. Its VF is $V\left(  u\right)  =u\left(  1+u\right)  ^{2}$ and
$\beta\left(  dx\right)  =\sum\nolimits_{n=0}^{\infty}\beta_{n}\delta
_{n}\left(  dx\right)  $ with $\beta_{n}=\left(  1+n\right)  ^{n-1}/n!$.
$\beta$ is induced by $g\left(  h\right)  =e^{h}$ and $h\left(  w\right)
=\sum_{n=0}^{\infty}\beta_{n}w^{n+1}$.
\end{example}

\begin{example}
Tak\'{a}cs family. Its VF is $V\left(  u\right)  =u\left(  1+u\right)  \left(
1+2u\right)  $ and $\beta\left(  dx\right)  =\sum\nolimits_{n=1}^{\infty}%
\beta_{n}\delta_{n}\left(  dx\right)  $ with $\beta_{n}=\frac{\left(
2n\right)  !}{n!\left(  n+1\right)  !}$. $\beta$ is induced by $g\left(
h\right)  =\left(  1-h\right)  ^{-1}$ with $h\left(  w\right)  =\sum
_{n=0}^{\infty}\beta_{n}w^{n+1}$.
\end{example}

\begin{example}
Large Arcsine. Its VF is $V\left(  u\right)  =u\left(  1+2u+2u^{2}\right)  $
and $\beta\left(  dx\right)  =\sum\nolimits_{n=0}^{\infty}\beta_{n}\delta
_{n}\left(  dx\right)  $ with $\beta_{n}=\frac{p_{n}\left(  1+n\right)
}{\left(  n+1\right)  !}$, where $p_{n}\left(  t\right)  $ is defined in
(\ref{eqPcoef}). $\beta$ is induced by $g\left(  h\right)  =\exp\left(
\arcsin h\right)  $ with $h\left(  w\right)  =\sum_{n=1}^{\infty}\beta
_{n}w^{n+1}$.
\end{example}

\subsection{RFs for NEFs with power variance functions\label{secNEFwithPVF}}

Let $D=\left\{  0\right\}  \cup\lbrack1,\infty)$. \cite{Lev:1986} showed that
for a $V\left(  u\right)  =au^{r}$ for $a>0$ and $r\in\mathbb{R}$ to be a VF
of some NEF, it is necessary that $r\in D$. They further showed that NEF-PVFs
are i.d. It should be pointed out that NEFs with PVFs include the Tweedie
family of \cite{Tweedie:1984}. It is easy to see that $U=\mathbb{R}$ iff $r=0$
and that $U=\mathbb{R}_{+}$ when $r\in D\setminus\left\{  0\right\}  $. For
more information on NEF-PVF, we refer the reader to \cite{Lev:1986}. Since
NEF-PVF with $r\in\mathbb{N}$ corresponds to a NEF-pVF discussed in
\autoref{ThmDiscreteN} and \autoref{secNEFwithCVF}, we will only show the
existence of $\varphi$ when $r\in D\setminus\mathbb{N}$.

\begin{corollary}
Suppose $\mathcal{F}$ is a NEF-PVF with VF $V\left(  u\right)  =au^{r}$ for
$r\in D\setminus\mathbb{N}$. Then $\alpha\ll\beta$ and $\varphi=\frac{d\alpha
}{d\beta}$.
\end{corollary}

\begin{proof}
WLOG, assume $a=1$. Then $\theta=\left(  1-r\right)  ^{-1}u^{1-r}$, $\theta<0$
and $u=\left(  \left(  1-r\right)  \theta\right)  ^{\frac{1}{1-r}}$. Let
$\gamma=\frac{2-r}{1-r}$. Set $C_{0,r}=\left(  2-r\right)  ^{-1}\left(
1-r\right)  ^{\gamma}$ and $C_{2,r}=\left(  1-r\right)  ^{\gamma-2}$.

\textbf{Case 1}: $r\in\left(  1,2\right)  $, i.e., $\gamma<0$. Then
$\kappa\left(  \theta\right)  =\left(  2-r\right)  ^{-1}u^{2-r}=C_{0,r}%
\theta^{\gamma}$ and $\kappa^{\prime\prime}\left(  \theta\right)
=C_{2,r}\theta^{\gamma-2}$. So,
\[
L\left(  \theta\right)  =\exp\left(  C_{0,r}\theta^{\gamma}\right)
=\sum_{n=0}^{\infty}\frac{1}{n!}C_{0,r}^{n}\theta^{n\gamma}%
\]
and%
\[
L\left(  \theta\right)  \kappa^{\prime\prime}\left(  \theta\right)
=C_{2,r}\theta^{\gamma-2}\sum_{n=0}^{\infty}\frac{1}{n!}\theta^{n\gamma
}C_{0,r}^{n}=\sum_{n=0}^{\infty}\frac{1}{n!}C_{2,r}C_{0,r}^{n}\theta
^{n\gamma+\gamma-2}.
\]
Since $\Gamma\left(  t\right)  =\int_{0}^{\infty}e^{-x}x^{t-1}dx$ for
$\Re\left(  t\right)  >0$, \ we see $\left(  -\theta\right)  ^{-t}=\frac
{1}{\Gamma\left(  t\right)  }\int_{0}^{\infty}x^{t-1}e^{\theta x}dx$.
Therefore, setting%
\[
q\left(  x\right)  =\sum_{n=1}^{\infty}\frac{\left(  -1\right)  ^{n\gamma
}C_{0,r}^{n}x^{-n\gamma-1}}{n!\Gamma\left(  -n\gamma\right)  }%
\]
and%
\[
\beta\left(  dx\right)  =\delta_{0}\left(  dx\right)  +q\left(  x\right)
1_{\left(  0,\infty\right)  }\left(  x\right)  dx
\]
gives $L\left(  \theta\right)  =\int_{0}^{\infty}e^{\theta x}\beta\left(
dx\right)  $. Similarly, setting $\alpha\left(  dx\right)  =\tilde{q}\left(
x\right)  dx$ with
\[
\tilde{q}\left(  x\right)  =\sum\nolimits_{n=0}^{\infty}\frac{\left(
-1\right)  ^{n\gamma+\gamma-2}C_{2,r}C_{0,r}^{n}x^{-n\gamma-\gamma+1}%
}{n!\Gamma\left(  -n\gamma-\gamma+2\right)  }1_{\left(  0,\infty\right)
}\left(  x\right)
\]
gives $L\left(  \theta\right)  \kappa^{\prime\prime}\left(  \theta\right)
=\int_{0}^{\infty}e^{\theta x}\alpha\left(  dx\right)  $. Further,
$\alpha=\beta\ast\rho$ satisfies $\alpha\ll\beta$ and $\varphi=\frac{d\alpha
}{d\beta}$.

It is easily seen that when $\gamma=-1$, i.e., $r=3/2$,
\[
\frac{d^{3}}{dx^{3}}\frac{x^{-n\gamma-\gamma+1}}{\Gamma\left(  -n\gamma
-\gamma+2\right)  }=\frac{x^{-n\gamma-1}}{\Gamma\left(  -n\gamma\right)  },
\]
which implies $q\left(  x\right)  =\left(  -1\right)  ^{\gamma-2}C_{2,r}%
\tilde{q}^{\left(  3\right)  }\left(  x\right)  $ and $\alpha\ll\beta$. So,
$h=\frac{d\alpha}{d\beta}$ when $r=3/2$. Explicitly, $C_{0,3/2}=-4$,
$C_{2,3/2}=-8$.
\[
q\left(  x\right)  =\delta_{0}\left(  dx\right)  +\sum_{n=1}^{\infty}%
\frac{4^{n}x^{-n-1}}{n!\left(  n-1\right)  !}1_{\left(  0,\infty\right)
}\left(  x\right)  dx
\]
and $\alpha\left(  dx\right)  =\sum\nolimits_{n=0}^{\infty}\frac{4^{n+2}}%
{n!}\frac{x^{n+2}}{\left(  n+2\right)  !}1_{\left(  0,\infty\right)  }\left(
x\right)  $. Clearly, $\alpha\ll\beta$, we have $\varphi=\frac{d\alpha}%
{d\beta}$.

\textbf{Case 2}: $r>2$, i.e., $0<\gamma<1.$ By the proof of Theorem 4.1 of
\cite{Lev:1986}, we can do the following. Let $G_{\gamma}$ be a stable
distribution with characteristic component $\gamma$ and let $\tilde{G}%
_{\gamma}\left(  x\right)  =G_{\gamma}\left(  x/\hat{a}\right)  $ for
$x\in\left(  0,\infty\right)  $, where $\hat{a}=\gamma^{-1/\gamma}\left(
1-\gamma\right)  ^{1/\gamma-1}$. Then $\tilde{G}_{\gamma}$ has Laplace
transform $L\left(  \theta\right)  =\exp\left(  -\left(  \hat{a}\theta\right)
^{\gamma}\right)  $ for $\theta\leq0$, $\kappa\left(  \theta\right)  =$
$C_{0,r}\left(  -\theta\right)  ^{\gamma}$ and $\kappa^{\prime\prime}\left(
\theta\right)  =C_{2,r}\left(  -\theta\right)  ^{\gamma-2}$. In fact, we can
set
\[
q\left(  x\right)  =-\frac{1}{\pi}\sum_{n=0}^{\infty}\frac{\left(  -1\right)
^{n}\sin\left(  n\pi\gamma\right)  }{n!}\frac{\Gamma\left(  1+\gamma n\right)
}{x^{1+\gamma n}}1_{\left(  0,\infty\right)  }\left(  x\right)  .
\]
and $\beta\left(  dx\right)  =\tilde{G}_{\gamma}\left(  dx\right)  =\left(
1/\hat{a}\right)  q\left(  x/\hat{a}\right)  dx$; see Section 5 of
\cite{Lev:1986}. Clearly,
\[
\rho\left(  dx\right)  =C_{2,r}\frac{x^{1-\gamma}}{\Gamma\left(
2-\gamma\right)  }1_{\left(  0,\infty\right)  }dx
\]
gives $\kappa^{\prime\prime}\left(  \theta\right)  =\int e^{\theta x}%
\rho\left(  dx\right)  $. So, $\alpha=\beta\ast\rho$ satisfies $\alpha\ll
\beta$ and $\varphi=\frac{d\alpha}{d\beta}$.
\end{proof}

\section{An application to estimating latent linear space\label{secApp}}

We provide an application of the RF $\varphi$ to estimating the latent space
in the means of high-dimensional data. For $k\gg n$ the data matrix
$\mathbf{Y}=\left(  y_{ij}\right)  \in\mathbb{R}^{k\times n}$ contains $n$
independent measurements of $k$ objects under study. Suppose the means
$\mathbb{E}\left[  y_{ij}\right]  $ are subject to a latent structure such
that%
\begin{equation}
\mathbb{E}\left[  \mathbf{Y}|\mathbf{M}\right]  =\left(  \mathbb{E}\left[
y_{ij}|\mathbf{M}\right]  \right)  =\mathbf{\Phi M} \label{eqModel}%
\end{equation}
for an unknown $\mathbf{M}\in\mathbb{R}^{r\times n}$ of rank $1\leq r\leq n$
and unknown $\mathbf{\Phi}\in\mathbb{R}^{k\times r}$ of rank $r$. In
(\ref{eqModel}), $\mathbf{M}$ can represent a design matrix or $n$
realizations of $r$ factors or $r$ latent variables. Then all $\mathbb{E}%
\left[  y_{ij}\right]  $ lie in the row space $\mathcal{S}_{\mathbf{M}}$ of
$\mathbf{M}$, and capturing $\mathcal{S}_{\mathbf{M}}$ is essential to
understanding the systematic variations in the data $\mathbf{Y}$. By Theorem 2
of \cite{Chen:2014intra}, $\mathcal{S}_{\mathbf{M}}$ can be consistently
estimated by the space spanned by the $r$ leading eigenvectors of%
\begin{equation}
\mathbf{G}_{k}=k^{-1}\mathbf{Y}^{T}\mathbf{Y}-\mathbf{D}_{k} \label{eqGramAdj}%
\end{equation}
as $k\rightarrow\infty$ with $n$ fixed, where $\mathbf{D}_{k}=\mathrm{diag}%
\left\{  \sigma_{i,k},\ldots,\sigma_{n,k}\right\}  $ and $\sigma_{i,k}%
=k^{-1}\sum\nolimits_{j=1}^{k}\mathbb{V}\left[  y_{ij}\right]  $. Note that
$\sigma_{i,k}$ is the average of the variances of $y_{ij}$'s in the $i$th
column of $\mathbf{Y}$ and that $\mathcal{S}_{\mathbf{M}}$ can not be
consistently estimated from $k^{-1}\mathbf{Y}^{T}\mathbf{Y}$ when
$\mathbf{D}_{k}$ is not a multiple of the identity matrix $\mathbf{I}_{n}$.
Therefore, we need to consistently estimate $\mathbf{D}_{k}$, i.e., to
estimate $\sigma_{i,k}$ for $1\leq i\leq n$.

If $y_{ij}$ follows a parametric distribution $F_{\theta_{ij}}$ for
$\theta_{ij}\in\Theta$ and $\mathcal{F}=\left\{  F_{\theta}:\theta\in
\Theta\right\}  $ forms an i.d. NEF with RF $\varphi$, then
\begin{equation}
\mathbb{E}\left[  \varphi\left(  y_{ij}\right)  \right]  =\mathbb{V}\left[
y_{ij}\right]  \label{eqIntra}%
\end{equation}
and $\hat{\sigma}_{i,k}=k^{-1}\sum\nolimits_{j=1}^{k}\varphi\left(
y_{ij}\right)  $ satisfies $\mathbb{E}\left[  \hat{\sigma}_{i,k}\right]
=\sigma_{i,k}$. In other words, the RF $\varphi$ induces an unbiased estimator
$\hat{\sigma}_{i,k}$ of $\sigma_{i,k}$, for which $\mathbf{\hat{D}}%
_{k}=\mathrm{diag}\left\{  \hat{\sigma}_{1,k},\ldots,\hat{\sigma}%
_{n,k}\right\}  $ is an unbiased estimator of $\mathbf{D}_{k}$. Under a moment
condition, it can be shown that $\lim_{k\rightarrow\infty}\left\vert
\hat{\sigma}_{i,k}-\sigma_{i,k}\right\vert =0$ for $1\leq i\leq n$ and
$\lim_{k\rightarrow\infty}\left\Vert \mathbf{\hat{D}}_{k}-\mathbf{D}%
_{k}\right\Vert =0$ hold almost surely; see Lemma 8 in \cite{Chen:2014intra}.
In view of this, $\varphi$ induces an unbiased estimator of $\mathbf{D}_{k}$
and enables consistent estimation of the latent space $\mathcal{S}%
_{\mathbf{M}}$ when $y_{ij}$'s follow a NEF we have identified.

\section*{Acknowledgements}

This research is funded by the Office of Naval Research grant
N00014-12-1-0764. I would like to thank John D. Storey for support, G\'{e}rard
Letac for guidance that has led to much improved results, Peter Jones for a
discussion on the maximal zero free domain of a Laplace transform, and Wayne
Smith for a discussion on the inverse of analytic functions.

\bibliographystyle{chicago}
%\bibliography{SpecialExpFam}

\end{document}